\documentclass[11pt]{amsart}

\usepackage[utf8x]{inputenc}
\usepackage[english]{babel}
\usepackage[all,cmtip]{xy}

\usepackage[pdftex]{graphicx}
\usepackage[margin=3cm]{geometry}

\usepackage{amsfonts}
\usepackage{amsmath}
\usepackage{amssymb}
\usepackage{amsthm}
\usepackage{dsfont}
\usepackage{tikz}
\usepackage{pgfplots}

\usepackage[T1]{fontenc}

\usepackage{paralist}

\usepackage[colorlinks,link color=blue, cite color=blue,pagebackref,pdftex]{hyperref}

\newcommand\Defn[1]{\textbf{\color{black}#1}}

\newcommand\eps{\varepsilon}
\renewcommand\emptyset{\varnothing}
\renewcommand\phi{\varphi}
               
\newcommand\Z{\mathbb{Z}}               
\newcommand\Znn{\Z_{\ge0}}
\newcommand\Q{\mathbb{Q}}               
\newcommand\R{\mathbb{R}}               

\newcommand\inner[1]{\langle{#1}\rangle}

\newcommand{\hosymbol}{
\begin{tikzpicture}
\draw[very thin] (0,0) circle (0.05cm);
\filldraw[black] (0,0)--(0,-0.05cm) arc (-90:90:0.05cm) -- (0,0);
\end{tikzpicture}
}

\newcommand\HO[1]{#1^{\hosymbol}}
\newcommand\half[2]{\mathbb{H}_{#1}{#2}}
\newcommand\Ihalf[2]{\mathbb{H}^\infty_{#1}{#2}}

\newcommand\Cay{\mathrm{Cay}}%

\renewcommand\l{\lambda}
\newcommand\LL{\Lambda}
\newcommand\PolL{\mathcal{P}(\LL)}
\newcommand\Pol[1]{\mathcal{P}(#1)}

\newcommand\vol{\mathrm{V}}
\newcommand\dvol{\mathrm{E}}

\DeclareMathOperator*{\relint}{relint}
\DeclareMathOperator*{\aff}{aff}
\DeclareMathOperator*{\conv}{conv}
\DeclareMathOperator*{\cone}{cone}

\makeatletter
\newtheorem*{rep@theorem}{\rep@title}
\newcommand{\newreptheorem}[2]{%
\newenvironment{rep#1}[1]{%
 \def\rep@title{#2 \ref{##1}}%
 \begin{rep@theorem}}%
 {\end{rep@theorem}}}
\makeatother

\newtheorem{thm}{Theorem}[section]
\newreptheorem{theorem}{Theorem}
\newtheorem*{thm*}{Theorem}
\newtheorem{cor}[thm]{Corollary}
\newtheorem{lem}[thm]{Lemma}
\newtheorem{prop}[thm]{Proposition}
\newtheorem{conj}{Conjecture}
\newreptheorem{conj}{Conjecture}

\theoremstyle{definition}

\newcommand\0{\boldsymbol{0}}
\renewcommand\1{\boldsymbol{1}}
\renewcommand\k{\mathbf{k}}
\newcommand\x{\mathbf{x}}
\newcommand\y{\mathbf{y}}
\newcommand\Zyl{\mathcal{Z}^{\hosymbol}}
\newcommand\Cls[1]{[\![{#1}]\!]}
\newcommand\Konv{\mathcal{K}}
\newcommand\Mix{\mathbf{M}}
\newcommand\CM{\mathbf{CM}}
\newcommand\MV{\mathrm{MV}}
\renewcommand\a{\alpha}

\newcommand\Npolys{r}

\title{Combinatorial Mixed Valuations}

\author{Katharina Jochemko}
\address{Department of Mathematics, %
Royal Institute of Technology (KTH), %
Sweden}
\email{jochemko@kth.se}
\author{Raman Sanyal}
\address{Institut für Mathematik, %
Goethe-Universität Frankfurt, %
Germany}
\email{sanyal@math.uni-frankfurt.de}

\keywords{combinatorial mixed valuations, mixed volumes,
translation-invariant valuations, combinatorial positivity, discrete 
mixed volume}
\subjclass[2010]{52B45, 05A10, 52B20, 52A39}

\date{\today}

\begin{document}

\maketitle

\begin{abstract}
    \emph{Combinatorial} mixed valuations associated to
    trans\-lation-invariant valuations on polytopes are introduced. In
    contrast to the construction of mixed valuations via polarization,
    combinatorial mixed valuations reflect and often inherit properties of
    inhomogeneous valuations. In particular, it is shown that under mild
    assumptions combinatorial mixed valuations are monotone and hence
    nonnegative.  For combinatorially positive valuations, this has strong
    computational implications. Applied to the discrete volume, the results
    generalize and strengthen work of Bihan (2015) on discrete mixed volumes.
    For rational polytopes, it is proved that combinatorial mixed monotonicity
    is equivalent to monotonicity.  Stronger even, a conjecture is
    substantiated that combinatorial mixed monotonicity implies the homogeneous
    monotonicity in the sense of Bernig--Fu (2011).
\end{abstract}

\section{Introduction} \label{sec:intro}
A momentous property of the $d$-dimensional Euclidean volume $\vol_d$ is that
for convex polytopes (and, more generally, convex bodies) $P_1,\dots,P_\Npolys
\subset \R^d$, the function $\vol_d(\l_1 P_1 + \cdots + \l_\Npolys P_\Npolys)$
agrees with a multivariate homogeneous polynomial of degree $d$ for all
$\l_1,\dots,\l_\Npolys \ge 0$.  For $\Npolys=d$, the coefficient of
$\l_1\l_2\cdots\l_d$, normalized by $\frac{1}{d!}$, is called the \Defn{mixed
volume} of $P_1,\dots,P_d$ and is denoted by $\MV_d(P_1,\dots,P_d)$. Mixed
volumes arise in virtually all mathematical disciplines and, most importantly,
give rise to the deep theory of geometric inequalities; see, for example,
Schneider~\cite{SchneiderNew}.  Among the most fundamental properties, one
trivially observes that $\MV_d$ is symmetric and Minkowski additive in each
argument and, not so trivially, that 
\begin{equation}\label{eqn:MV_mon}
    0 \ \le \ \MV_d(P_1,\dots,P_d) \ \le \ \MV_d(Q_1,\dots,Q_d)
\end{equation}
for all polytopes $P_i \subseteq Q_i$ for $i = 1,\dots,d$.

For the \Defn{discrete volume} $\dvol(P) := |P \cap \Z^d|$,
Bernstein~\cite{Bernstein} and McMullen~\cite{McMullen74,mcmullenEuler} showed that for
polytopes $P_1,\dots,P_\Npolys \subset \R^d$ with vertices in the lattice
$\Z^d$, the function $\dvol_{P_1,\dots,P_\Npolys}(n_1,\dots,n_\Npolys) =
\dvol(n_1P_1 + \cdots + n_\Npolys P_\Npolys)$ also agrees with a multivariate
polynomial---the multivariate Ehrhart polynomial---for all $n_1,\dots,n_\Npolys \in
\Znn$. As for the volume, this result sets the stage for a \emph{mixed}
Ehrhart theory; see, for example,~\cite{HKST, KatzStapledon,
SteffensTheobald}.  Suitable polarizations of $\dvol_{P_1,\dots,P_\Npolys}$
give rise to a \emph{discrete} counterpart to the mixed volume. By
construction $\Mix\dvol(P_1,\dots,P_\Npolys)$ is also symmetric and Minkowski
additive but the nonnegativity and monotonicity properties~\eqref{eqn:MV_mon}
are genuinely lost. This is due to the fact that the discrete volume, unlike
$\vol_d$, is \emph{not} a homogeneous valuation and thus cannot be treated as
such.

The purpose of this paper is to introduce the notion of a \emph{combinatorial}
mixed valuation that extends the notion of mixed volume with many of its
favorable properties to the class of $\LL$-valuations: Let $\LL \subseteq
\R^d$ be a lattice or vector subspace over a subfield of $\R$ and write
$\PolL$ for the collection of polytopes with vertices in $\LL$. A
\Defn{$\LL$-valuation} is a map $\phi : \PolL \rightarrow G$ taking values in
an abelian group $G$ such that $\phi(\emptyset) = 0$ and 
\begin{equation}\label{eqn:val}
    \phi(P \cup Q) \ = \ \phi(P) + \phi(Q) - \phi(P\cap Q),
\end{equation}
for all $P,Q \in \PolL$ such that with $P\cup Q, P \cap Q \in \PolL$ and
$\phi(P + t) = \phi(P)$ for all $t \in \LL$. McMullen~\cite{mcmullenEuler}
showed that for a $\LL$-valuation $\phi$, the map $\phi_P(n) := \phi(nP)$
agrees with a polynomial of degree at most $\dim P$. Using the fact that
$\phi^{+Q}(P) := \phi(P+Q)$ is a $\LL$-valuation for fixed $Q \in \PolL$, it
can be shown (see, e.g.~\cite{JS15}) that 
\[
    \phi_{P_1,\dots,P_\Npolys}(n_1,\dots,n_\Npolys) \ := \ \phi(n_1P_1 +
    \cdots + n_\Npolys P_\Npolys)
\]
agrees with a multivariate polynomial. 
For integers $\Npolys \ge 0$, let $[\Npolys]$ denote the set $\{1,2,\ldots, r\}$. In particular, $[0]=\emptyset$. We define the \Defn{$\Npolys$-th
combinatorial mixed valuation} associated to $\phi$ by
\begin{equation}\label{eqn:DM}
    \CM_\Npolys\phi(P_1,\dots,P_\Npolys) \ := \ \sum_{I \subseteq [\Npolys]}
    (-1)^{\Npolys - |I|}
    \phi(P_I),
\end{equation}
for $P_1,\dots,P_\Npolys \in \PolL$ and where $P_I := \sum_{i \in I} P_i$ is
the Minkowski sum and $P_\emptyset := \{0\}$. By convention $\CM_0\phi =
\phi(\{0\})$ and $\CM_\Npolys\phi(P_1,\dots,P_\Npolys) = 0$ for all choices of
$\Npolys > d$ polytopes; see Corollary~\ref{cor:zero}.  We drop the index
$\Npolys$ and simply write $\CM\phi$ when no confusion arises.  Clearly,
$\CM_\Npolys\phi$ is symmetric and a $\LL$-valuation in each of its arguments.
In Theorem~\ref{thm:repr} we show that, like the mixed volume, the
combinatorial mixed volume can be characterized by a universal property.

For $\phi = \vol_d$ and $r=d$, our definition recovers the usual mixed volume
$\CM\vol_d(P_1,\dots,P_d)  = d! \, \MV_d(P_1,\dots,P_d)$ and it was shown by
Bernstein~\cite{Bernstein} that if $P_1,\dots,P_d \subset \R^d$ are lattice polytopes, then 
\[
    \CM\dvol(P_1,\dots,P_d) \ = \ \sum_{I \subseteq [d]} (-1)^{d-|I|} E(P_I) \
    = \ d! \, \MV_d(P_1,\dots,P_d);
\]
see~\cite[Cor.~2.3]{HKST}. For $r < d$, $\CM\dvol(P_1,\dots,P_\Npolys)$ was
investigated
by Bihan~\cite{Bihan} under the name of \emph{discrete mixed volume} in the
context of fewnomial bounds and tropical intersection theory. In particular,
using \emph{irrational mixed decompositions} and an ingenious but involved
argument, Bihan showed that the discrete mixed volume is always nonnegative.
Assume that the value group $G$ is partially ordered.  We call a
$\LL$-valuation $\phi : \PolL \rightarrow G$ \Defn{combinatorially mixed
monotone} (or \Defn{CM-monotone}, for short) if 
\begin{equation}\label{eqn:CMmono}
    \CM\phi(P_1,\dots,P_\Npolys) \ \le \ \CM\phi(Q_1,\dots,Q_\Npolys)
\end{equation}
for all $\Npolys \ge 0$ and $\LL$-polytopes $P_i \subseteq Q_i$ for
$i=1,\dots,\Npolys$.  Setting $Q_i = \{0\}$ for $i=1,\dots,\Npolys$ shows that
$\CM\phi(P_1,\dots,P_\Npolys) \ge 0$ for all $\Npolys$. In this paper, we
study the class of CM-monotone valuations in relation to the classes
introduced in~\cite{JS15}. Our main results give strong sufficient conditions
for CM-monotonicity.

In~\cite{JS15}, we introduced the notion of \Defn{weakly $h^*$-monotone
valuations} that are characterized by the property that $\phi(\relint S) +
\phi(\relint F) \ge 0$ for all $\LL$-simplices $S$ and facets $F \subset S$.
As is customary
\[
    \phi(\relint S ) \ = \ \sum_F (-1)^{\dim S - \dim F} \phi(F),
\]
where the sum is over all faces $F$ of $S$. Our first result is the following.
\begin{thm}\label{thm:CM_weakly}
    Let $\phi : \PolL \rightarrow G$ be a $\LL$-valuation with values in a
    partially ordered group $G$. If $\phi$ is weakly $h^*$-monotone, then
    $\phi$ is CM-monotone.
\end{thm}

Since the discrete volume is nonnegative on relative interiors (that is,
$\dvol$ is combinatorially positive), Theorem~\ref{thm:CM_weakly} yields  a
strengthening of Bihan's result~\cite[Thm.~1.2(2)]{Bihan}.
\begin{cor}\label{cor:Bihan}
    The discrete mixed volume $\CM\dvol(P_1,\dots,P_\Npolys)$ is monotone and
    hence nonnegative.
\end{cor}

Di~Rocco, Haase, and Nill~\cite{DHN} interpret $\CM\dvol(P_1,\dots,P_\Npolys)$
as the \emph{motivic arithmetic genus} of a generic complete intersection with
prescribed Newton polytopes $P_1,\dots,P_\Npolys$.  CM-monotonicity implies
that the motivic arithmetic genus is monotone with respect to inclusion of
Newton polytopes.  Under the stronger assumption that $\phi$ is
combinatorially positive, we give lower bounds on $\CM\phi$ in
Section~\ref{ssec:bounds}.  Techniques from~\cite{DGH} generalize and yield
that checking whether $\CM\dvol$ is positive can be done in polynomial time.
Our proof of Theorem~\ref{thm:CM_weakly}, given in Section~\ref{ssec:cones},
casts the statement into the language of cones in McMullen's polytope
algebra~\cite{mcmullenPA}. In order to further popularize the polytope algebra
we give a brief, tailor-made introduction. For the discrete volume,
Corollary~\ref{cor:Bihan} can also be obtained by considering Cayley cones. We
sketch the proof in Section~\ref{ssec:cayley}.

In Section~\ref{sec:monotone} we show that there are strong relations between
\emph{monotonicity} and \emph{CM-monotonicity}.  In particular, we prove the
following result.
\begin{thm}\label{thm:Rd-mono}
    An $\R^d$-valuation $\phi : \Pol{\R^d} \rightarrow \R$ is monotone if and
    only if $\phi$ is CM-monotone.
\end{thm}

This is a combinatorial analog of the following deep result of Bernig and
Fu~\cite{BernigFu}.
\begin{thm}[{\cite[Thm.~2.12]{BernigFu}}]\label{thm:BF}
    Let $\phi$ be a translation-invariant valuation on convex bodies in $\R^d$
    and $\phi = \phi_0 + \cdots + \phi_d$ the decomposition into homogeneous
    components. Then $\phi$ is monotone if and only if $\phi_i$ is monotone
    for all $i=0,\dots,d$.
\end{thm}
We deduce Theorem~\ref{thm:Rd-mono} from a multivariate version of this result
(Lemma~\ref{lem:bernig}) and results of~\cite{HKST}. In fact,
Theorem~\ref{thm:Rd-mono} and the trivial observation that CM-monotone implies
monotone prompted the following conjecture.
\begin{conj}\label{conj:mono}
    Let $\phi$ be a $\Lambda$-valuation for any $\LL$. Then $\phi$ is monotone
    if and only if $\phi$ is CM-monotone.
\end{conj}

To support Conjecture~\ref{conj:mono}, we proof that, if true, it implies
Theorem~\ref{thm:BF}.

\begin{thm}\label{thm:CM-BF}
    Conjecture~\ref{conj:mono} implies Theorem~\ref{thm:BF}.
\end{thm}

\textbf{Acknowledgements.} We would like thank Christian Haase and Monika
Ludwig for inspiring discussions and Andreas Bernig for his interest and his
help with Lemma~\ref{lem:bernig}. We also thank Benjamin Nill, Matthias
Schymura, Thorsten Theobald and the anonymous referee for helpful comments.
R.~Sanyal was supported by the DFG-Collaborative Research Center, TRR 109
``Discretization in Geometry and Dynamics''.

\section{Algebraic characterization of combinatorial mixed valuations}
\label{sec:algebra}
\newcommand\Diff{\triangle}%

The mixed volume $\MV_d(P_1,\dots,P_d)$ is the unique symmetric and
(positively) multilinear form such that $\MV_d(P,\dots,P) = \vol_d(P)$;
cf.~\cite[Thm.~5.1.7]{SchneiderNew}. In this section, we will give an
analogous algebraic characterization of combinatorial mixed valuations.  A
function $f\colon \Z_{\ge 0}^\Npolys \rightarrow G$ is a polynomial of degree
$\le d$ if $d = 0$ and $f(n_1,\dots,n_\Npolys)$ is constant or 
\[
    \Diff_i f \ := \  f(n_1,\dots,n_i + 1,\dots,n_\Npolys) -
    f(n_1,\dots,n_i,\dots,n_\Npolys) 
\]
is a polynomial of degree $< d$ for all $i = 1,2,\dots,\Npolys$. As alluded to in the introduction,
McMullen~\cite{mcmullenEuler} showed the following result that underlies most
of the theory of translation-invariant valuations.
\begin{thm}\label{thm:poly}
    Let $\phi : \PolL \rightarrow G$ be a $\LL$-valuation with values in an
    abelian group $G$. Then for any polytopes $P_1,\dots,P_\Npolys \in \PolL$,
    the function
    \[
        \phi_{P_1,\dots,P_\Npolys}(n_1,\dots,n_\Npolys) \ := \ 
        \phi(n_1P_1 + \cdots + n_\Npolys P_\Npolys)
    \]
    is a polynomial of degree $\le \dim P_1 + \cdots + P_\Npolys$.
\end{thm}

Appealing to the \emph{calculus of finite
differences}~\cite[Sect.~1.9]{EC1new}, we obtain a more familiar
representation of a polynomial $f(n_1,\dots,n_\Npolys)$ as
\newcommand\Ba{{\boldsymbol\alpha}}
\begin{equation}\label{eqn:poly}
    f(n_1,\dots,n_\Npolys) \ = \ \sum_{\Ba \in \Znn^\Npolys} \Diff^\Ba f(\0) \,
    \binom{n_1}{\alpha_1}\cdots \binom{n_\Npolys}{\alpha_\Npolys},
\end{equation}
where $\Diff^\Ba f := \Diff_1^{\alpha_1}\cdots \Diff_\Npolys^{\alpha_\Npolys}
f$. Here, $\Ba=(\alpha _1,\ldots,\alpha _\Npolys) \in
\Znn^\Npolys$ serves as a multi-index and we define $|\Ba|=\alpha _1+\cdots +\alpha _\Npolys$. Indeed, let $\hat{f}(n_1,\dots,n_\Npolys)$ denote the right-hand side
of~\eqref{eqn:poly}.  Then $\hat{f}$ is polynomial in $n_1,\dots,n_\Npolys$ of
total degree $\le d$ and using the fact that $\Diff_i \binom{n_j}{k} = 0$ if
$i \neq j$ and $=\binom{n_j }{k-1}$, otherwise, it is straightforward to
verify that $\Diff^\Ba \hat{f} = \Diff^\Ba f$ for all $\Ba \in \Znn^\Npolys$.

For the polynomial $\phi_{P_1,\dots,P_\Npolys}(n_1,\dots,n_\Npolys)$ we note that
\[
    \Diff_1\cdots\Diff_\Npolys \phi_{P_1,\dots,P_\Npolys}(\0) \ = \ \sum_{I
    \subseteq [\Npolys]}
    (-1)^{\Npolys - |I|} \phi(P_I) \ = \ \CM\phi(P_1,\dots,P_\Npolys).
\]
Thus, setting 
\[
    \CM\phi(P_1^{\alpha_1},\dots,P_\Npolys^{\alpha_\Npolys}) \ := \
    \CM\phi(\underbrace{P_1,\dots,P_1}_{\alpha_1},  \dots,
    \underbrace{P_\Npolys,\dots,P_\Npolys}_{\alpha_\Npolys})
\]
we conclude the following defining property of combinatorial mixed valuations.
\begin{thm}\label{thm:repr}
    Let $\phi$ be a $\LL$-valuation and $P_1,\dots,P_\Npolys \in \PolL$. Then
    \[
        \phi(n_1P_1+\cdots+n_\Npolys P_\Npolys) \ = \ 
        \sum_{\Ba \in \Z_{\ge 0}^\Npolys}
        \CM\phi(P_1^{\a_1},\dots,P_\Npolys^{\alpha_\Npolys})
        \binom{n_1}{\alpha_1}\cdots
        \binom{n_\Npolys}{\alpha_\Npolys}.
    \]
    In particular, for all $\Ba \in \Z_{\ge 0}^\Npolys$,
    $\CM\phi(P_1^{\a_1},\dots,P_\Npolys^{\alpha_\Npolys})$ is a valuation in every
    argument $P_1,\dots, P_\Npolys$.
\end{thm}

Together with Theorem~\ref{thm:poly} this yields  the following.
\begin{cor}\label{cor:zero}
    Let $\phi$ be a $\LL$-valuation on $\R^d$ and let $P_1,\dots,P_\Npolys$ be
    $\LL$-polytopes. The combinatorial mixed valuation
    $\CM\phi(P_1,\dots,P_\Npolys)$ is the coefficient of
    $\binom{n_1}{1}\cdots\binom{n_\Npolys}{1}$ of the polynomial $\phi(n_1P_1 + \cdots +
    n_\Npolys P_\Npolys)$.
    Moreover, $\CM_\Npolys\phi \equiv 0$
    for $\Npolys > d$.
\end{cor}

A remarkable property is that in the binomial basis, the coefficients of a
polynomial $f(n_1,\dots,n_\Npolys)$ are unique and independent of the
coefficient group $G$. This observation together with~\eqref{eqn:poly} yields
a characterization of combinatorial mixed valuations. 
\begin{cor}\label{cor:charac}
    Let $\phi : \PolL \rightarrow G$ be a $\LL$-valuation. Let $D = (D_r :
    \PolL^r \rightarrow G)_{r \ge 0}$ be a family of maps such that $D_r$ is
    symmetric and a $\LL$-valuation in each argument and 
    \begin{enumerate}[\rm (i)]
        \item $D_0 = \phi(\{0\})$,
        \item $D_1(P) = \phi(P) - D_0$, and
        \item for all $\Npolys\geq 2$ and for any $P_1,P_2,\dots,P_{\Npolys} \in \PolL$
            \[
                D_\Npolys(P_1,P_2,\dots,P_{\Npolys}) \ = \
                D_{\Npolys-1}(P_1+P_2,P_3,\dots,P_{\Npolys}) -
                D_{\Npolys-1}(P_1,P_3,\dots,P_{\Npolys}) -
                D_{\Npolys-1}(P_2,\dots,P_{\Npolys}).
            \]
    \end{enumerate}
    Then $D = \CM\phi$.
\end{cor}

\section{Combinatorial mixed valuations and the polytope algebra}
\label{sec:CM_PiL}

\subsection{Cones in the Polytope Algebra}
\label{ssec:cones}

We wish to cast the statement of Theorem~\ref{thm:CM_weakly} into a more
conceptual setting, namely, that of cones in the polytope
algebra~\cite{mcmullenPA}. As the
polytope algebra is not as well known as it should be, we start from scratch.
For a fixed $\LL$, let $\Z\PolL$ be the free abelian group with generators
$e_P$ for $P \in \PolL$. Consider the subgroup $U$ that is generated by
elements of the form
\begin{equation}\label{eqn:PiL}
\begin{aligned}
    e_{P \cup Q} \ + \ e_{P\cap Q} \ - \ e_P \ - \ e_Q \quad & \quad \text{ for
$P,Q \in \PolL$
    with $P\cup Q, P \cap Q \in \PolL$, and } \\
    e_{P + t} \ - \ e_P \quad & \quad \text{ for $P\in \PolL$ and $t \in
    \LL$.}
\end{aligned}
\end{equation}
\newcommand\PiL{\Pi(\LL)}%
The \Defn{polytope algebra} is the quotient $\PiL := \Z\PolL / U$. We write
$\Cls{P}$ for $e_P + U \in \PiL$.  The map $\PolL \rightarrow \PiL$ given by
$P \mapsto \Cls{P}$ is the \emph{universal} $\LL$-valuation in the following
sense. 

\newcommand\extPhi{\overline{\phi}}
\begin{prop}\label{prop:universal}
    Let $G$ be an abelian group.  For every $\LL$-valuation $\phi : \PolL
    \rightarrow G$ there is a unique homomorphism of abelian groups $\extPhi :
    \PiL \rightarrow G$ such that the following diagram commutes:
    \[  
        \xymatrix{ 
                \PolL \ar@{->}[d] \ar@{->}[r]^\phi & G 
                \\ \PiL \ar@{-->}[ru]_{\extPhi}
        }.
    \]
    In particular, $\hom(\PiL,G)$ is the group of $\LL$-valuations on $\PolL$
    with values in $G$.
\end{prop}

A valuation $\phi$ has the \Defn{inclusion-exclusion property} if for any
collection of $\LL$-polytopes $P, P_1,\dots,P_\Npolys$ such that $P = P_1 \cup
\cdots \cup P_\Npolys$ and $\bigcap_{i \in I} P_i \in \PolL$ for
every $I \subseteq [\Npolys]$,
the following holds:
\begin{equation}\label{eqn:IE}
    \phi(P) \ = \ \sum_{I \subseteq [\Npolys]} (-1)^{\Npolys-|I|}
    \phi(\bigcap_{i \in I}
    P_i).
\end{equation}

A priori, it is \emph{not} clear that every $\LL$-valuation has the
inclusion-exclusion property. Volland~\cite{Volland} showed that the
inclusion-exclusion property holds whenever $\LL$ is a vector space over a
subfield of $\R$. Betke (unpublished) and in a stronger form
McMullen~\cite{McMullenIE} verified this when $\LL$ is a lattice\footnote{More
precisely, Volland and Betke-McMullen showed the inclusion-exclusion property
without the assumption of translation-invariance.}. We record their results in
abstract form.
\begin{thm}\label{thm:IE}
    The universal valuation $\PolL \rightarrow \PiL$ (and hence every
    valuation) has the inclusion-exclusion property for any $\LL$.
\end{thm}

It can be verified that $\PiL$ is a commutative ring with unit $1 =
\Cls{0}$ with respect to multiplication given by Minkowski addition:
$\Cls{P}\cdot\Cls{Q} = \Cls{P+Q}$.  For $\LL = \R$ (or more generally a vector
space over an ordered  field), the polytope algebra was thoroughly
investigated by McMullen~\cite{mcmullenPA} (building on work of Jessen and
Thorup~\cite{JT}) and by Morelli~\cite{morelli}. Universality together with 
Theorem~\ref{thm:repr} implies $(\Cls{P}-1)^{\dim P + 1} = 0$ for all
$\LL$-polytopes $P \subseteq \R^d$.

\begin{cor}\label{cor:CMcls}
    For polytopes $P_1,\dots,P_\Npolys \in \PolL$
    \begin{equation}\label{eqn:polynom_Pi}
    \Cls{n_1P_1 + \cdots + n_\Npolys P_\Npolys} \ = \ \sum_{\a} \CM\Cls{
    P_1^{\a_1},\dots,
    P_\Npolys^{\a_\Npolys}} 
    \binom{n_1}{\a_1}
    \cdots
    \binom{n_\Npolys}{\a_\Npolys},
\end{equation}
where 
\begin{equation}\label{eqn:CMcls}
    \CM\Cls{ P_1,\dots, P_r} \ := \ \sum_{I \subseteq [r]} (-1)^{r - |I|}
    \Cls{P_I} \ = \ \prod_{i=1}^r (\Cls{P_i} - 1)
\end{equation}
is the discrete mixed valuation associated to the universal valuation $\PolL
\rightarrow \PiL$. In particular, if $\phi : \PolL \rightarrow G$ is a
$\LL$-valuation, then $\CM\phi(P_1,\dots,P_\Npolys) =
\extPhi(\CM\Cls{P_1,\dots,P_\Npolys})$ in the sense of
Corollary~\ref{cor:charac}.
\end{cor}

Properties of $\LL$-valuations considered in Section~\ref{sec:intro} can
be phrased in the language of cones in $\PiL$. To begin with, we define the
\Defn{monotone cone} as the set
\newcommand\M{\mathcal{M}}%
\begin{equation}\label{eqn:monotone_cone}
    \M(\LL) \ := \ \Z_{\ge 0}\{ \Cls{Q} - \Cls{P} : P \subseteq Q \text{
    $\LL$-polytopes} \} \ \subseteq \ \PiL.
\end{equation}
We will simply write $\M$ if $\LL$ and hence $\Pi = \PiL$ is clear from the
context.  Now, for a partially ordered group $(G,\preceq)$ and a subset $C
\subseteq \Pi$, we define
\[
    \hom_+(C,G) \ := \ \{ \phi \in \hom(\Pi,G) : \phi(x) \succeq 0 \text{ for
    all } x \in C \}.
\]
In this language, we note that
\begin{prop}\label{prop:monotone_cone}
    Let $G$ be a partially ordered abelian group. A $\LL$-valuation $\phi :
    \PolL \rightarrow G$ is monotone if and only if $\phi \in \hom_+(\M,G)$.
\end{prop}

This allows us to phrase CM-monotonicity as a sort of mixed or higher
monotonicity property: For $\Npolys \ge 0$
\[
    \M_\Npolys(\LL) \ := \ \Z_{\ge 0}\left\{ \prod_{i=1}^\Npolys (\Cls{Q_i} -
    1) - \prod_{i=1}^\Npolys
    (\Cls{P_i} - 1) : \text{ for $\LL$-polytopes } P_i \subseteq Q_i \text{
    for } i=1,\dots,\Npolys \right\} .
\]
Observe that $\M_\Npolys = 0$ whenever $\Npolys = 0$ or $\Npolys \ge d+1$ and
$\M  = \M_1$.  We define the \Defn{mixed monotone cone} as
\newcommand\MM{\overline{\M}}
\[
    \MM \ := \ \M_1 + \M_2 + \cdots + \M_d
\]
and conclude that $\hom_+(\MM,G)$ are precisely the CM-monotone $\LL$-valuations
taking values in $G$.

\newcommand\CP{\mathcal{C}}%
In particular, the results from~\cite{JS15} can be expressed in this language.
For a polytope $P \in \PolL$ we define
\[
    \Cls{\relint P} \ := \ \sum_{F} (-1)^{\dim P - \dim F} \Cls{F},
\]
where the sum is over all faces of $P$. This ensures that $\Cls{P} = \sum_F
\Cls{\relint F}$.  The \Defn{combina\-torially-positive cone} $\CP$ is the
semigroup spanned by the classes $\Cls{\relint{S}}$ where $S$ ranges over all
$\LL$-simplices and  the \Defn{weak $h^\ast$-monotone cone} is defined as
\newcommand\WCM{\mathcal{W}}%
\[
    \WCM \ := \ \Z_{\ge 0}\left\{ \Cls{\relint{S}} + \Cls{\relint{F}} : S \in
    \PolL \text{ simplex} , F \subset S \text{ facet} \right\}.
\]
The statements of Theorem~3.6 and Theorem~5.1 in~\cite{JS15} can now be
phrased in terms of cones in $\PiL$.
\begin{thm}
    Let $P$ be a $\LL$-polytope of dimension $r$. Then
    \[
        \Cls{n P} \ = \ \sum_{i=0}^r c_i\binom{n+r-i}{r}
    \]
    where $c_0,\dots,c_r \in \CP$. Let $P \subseteq Q$ be two $\LL$-polytopes
    of dimension $r$. Then
    \[
        \Cls{n Q} - \Cls{n P} \ = \ \sum_{i=0}^r w_i\binom{n+r-i}{r}
    \]
    where $w_0,\dots,w_r \in \WCM$.
\end{thm}
In light of the fact that the $h^\phi$-vector of a $\LL$-polytope $P$ is
defined through
\[
    \phi(nP) \ = \ \sum_{i=0}^r h^\phi_i(P) \binom{n+r-i}{r}
\]
it can be easily seen that $\phi$ is weakly $h^*$-monotone if and only if
$\phi \in \hom_+(\WCM,G)$ and $h^*$-nonnegative if and only if $\phi \in
\hom_+(\CP,G)$.

\subsection{Half-open simplices and cylinders}\label{ssec:hopen_zyl}
To a nonempty polytope $P \subset \R^d$ and a point $q \in \aff(P)$ we can
associate a \Defn{half-open} polytope $\half{q}{P}$ as the set of points $p
\in P$ such that $[q,p) \cap P \neq \emptyset$. That is, $\half{q}{P}$ is obtained by removing all facets of $P$ that are visible from $q$. Note that $\half{q}{P}$ is
never an open set and $\half{q}{P}$ is closed if and only if $q \in P$. We
call $\half{q}{P}$ \Defn{properly} half-open if $q \not\in P$. We will simply
write $\HO{P}$ for $\half{q}{P}$ if we do not want to specify $q$.

In~\cite{JS15} we extensively used half-open polytopes together with
dissections to reduce our results to questions about half-open simplices by
making use of the following fact.
\begin{lem}[{\cite[Thm.~3]{KV}}]\label{lem:ho}
    Let $P$ be a polytope, $P = P_1 \cup \cdots \cup P_\Npolys$ a dissection
    and $q \in \aff(P)$ a general point. Then
    \[
        \half{q}{P} \ = \ 
        \half{q}{P_1} \uplus 
        \cdots \uplus 
        \half{q}{P_\Npolys}
    \]
    is a disjoint union of half-open polytopes.  If $\half{q}{P}$ is properly
    half-open, then all $\half{q}{P_i}$ are properly half-open.
\end{lem}

A drawback of this notion of half-open polytopes is that in general
$t + \half{q}{P} \neq \half{q}{(t+P)}$ or, equivalently, $\half{q}{P} \neq
\half{q-t}{P}$. The following proposition tries to remedy this. We call a
vector $u \in \R^d\setminus \{0\}$ \Defn{general} with respect to a polytope
$P$ if the line $\R u$ is parallel to $\aff(P)$ but not parallel to any
facet-defining hyperplane.

\begin{prop}\label{prop:ihalf}
    Let $P \subset \R^d$ be a nonempty polytope and $u$ general with respect
    to $P$. For any two points $q_1,q_2 \in \aff(P)$ we have
    \[
        \half{q_1 + \mu_1 u}{P} \ = \
        \half{q_2 + \mu_2 u}{P}.
    \]
    for all sufficiently large $\mu_1,\mu_2 \gg 0$.
\end{prop}
We denote this half-open polytope by $\Ihalf{u}{P}$. The half-open polytope is
obtained by removing all facets of $P$ that are visible `from infinity' in
direction $u$.

\begin{proof}
    Let $P = \{ x \in \R^d : \inner{a_i,x} \le b_i \text{ for } i =
    1,\dots,m\}$  such that $F_i = \{ x \in P : \inner{a_i,x} = b_i\}$ is a
    facet for all $i=1,\dots,m$. For a point $q$ we define $I(q) := \{ i :
    \inner{a_i,q} > b_i \}$. Then 
    \[
        \half{q}{P} = P \setminus \bigcup_{i \in I(q)} F_i.
    \]
    Since $\inner{a_i, u} \neq 0$ for all $i$, it follows that for $\mu \gg 0$
    sufficiently large $I(q+\mu u) = \{ i : \inner{a_i,u} > 0\}$ and the
    result follows.
\end{proof}

Note that it is generally not true that for any $q$ there is a $u$ such that
$\half{q}{P} = \Ihalf{u}{P}$. However, this holds if $P$ is a simplex. This
follows directly from the proof of Proposition~\ref{prop:ihalf} and the fact
that a simplex has affinely independent facet normals.
\begin{prop}\label{prop:simplex_ihalf}
    Let $S \subset \R^d$ be a $d$-simplex. Then for any $q \not \in S$ there
    is a $u \neq 0$ such that $\half{q}{S} = \Ihalf{u}{S}$.
\end{prop}

We call a Minkowski sum $P = P_1 + \cdots + P_\Npolys$ \Defn{exact} if 
\[
    \dim P  \ = \ \dim P_1 + \cdots + \dim P_\Npolys
\]
and we call $P$ a \Defn{cylinder} if moreover $P_1,\dots,P_\Npolys$ are simplices.
Finer even, $P$ is a \Defn{$k$-cylinder} if exactly $k$ of the
$\Npolys$
simplices are of positive dimension.

\begin{prop}\label{prop:prod_half}
    Let  $S = S_1 + \cdots + S_\Npolys$ be a cylinder and $q \in \aff(S)$. Then
there are $q_i \in \aff(S_i)$ for $i=1,\dots,\Npolys$ such
    that 
    \[
        \half{q}{S} \ = \ \half{q_1}{S_1} + \cdots +
        \half{q_\Npolys}{S_\Npolys}.
    \]
    Consequently, if $q\not \in S$, then there is a $u \neq 0$ parallel to
$\aff(S)$ such that
    $\Ihalf{u}{S} = \half{q}{S}$.
\end{prop}
\begin{proof}
    Observe that if $T : \R^d \rightarrow \R^d$ is an affine transformation,
    then $\half{Tq}{TS} = T\half{q}{S}$. Exactness means that $S$ is affinely
    isomorphic to a Cartesian product $S_1 \times \cdots \times S_\Npolys$ and
    writing $q = (q_1,\dots,q_\Npolys) \in \aff(S) = \aff(P_1) \times \cdots \times
    \aff(P_\Npolys)$ proves the first claim. The second claim follows directly
    from Proposition~\ref{prop:simplex_ihalf}.
\end{proof}

For $k = 0,\dots,d$, we define the \Defn{cone of half-open
$k$-cylinders}
\[
    \Zyl_k \ := \ \Z_{\ge 0}\{ \Cls{S} : S \text{ half-open
    $k$-cylinder} \}.
\]
By definition $\Zyl_0 = \Znn\{ \Cls{0} = 1 \}$.  For the proof of the next
result, recall that a \Defn{dissection} of an $l$-polytope $P \subset \R^d$ is
a collection $P_1,\dots,P_\Npolys \subseteq P$ of $l$-dimensional polytopes
such that $P = P_1 \cup \cdots \cup P_\Npolys$ and $\dim P_i \cap P_j < l$ for
all $i \neq j$.

\begin{lem}\label{lem:Zyl_chain}
    \[
        \Zyl_d \ \subseteq \ \Zyl_{d-1} \ \subseteq \ \cdots \ \subseteq \  \Zyl_1 \ = \
        \WCM.
    \]
\end{lem}
\begin{proof}
    Let $S = S_1 + S_2 + \cdots + S_{\Npolys}$ be an $\Npolys$-cylinder with $\Npolys \ge 2$.
    Choose a dissection 
    \[
        S_1 + S_2 \ = \ T_1 \cup \cdots \cup T_m
    \]
    into simplices coming, for example, from the \emph{staircase
    triangulation}~\cite[Sect.~6.2.3]{DLRS} of $S_1 + S_2 \cong S_1 \times S_2$.
    Since $\aff(S_1+S_2) = \aff(T_i)$ for all $1 \le i \le m$, it follows that 
    \[
        S \ = \ \bigcup_{i = 1}^m (T_i + S_3 + \cdots + S_\Npolys)
    \]
    is a dissection of $S$ into $(\Npolys-1)$-cylinders. Applying
    $\half{q}{\cdot}$ to both sides for some point $q$ yields a partition of
    $\half{q}{S}$ into half-open $(\Npolys-1)$-cylinders.

    For the equality, note that we only have to show that $\Zyl_1 \subseteq
    \WCM$. Let $S$ be a half-open $s$-simplex. Let us write $M(S)\geq 1$ for
    the number of facets that were not removed from the underlying closed
    simplex $\bar{S}$. If $M(S) = 1$, then $\Cls{S} \in \WCM$ by definition.
    If $M(S) > 1$, then let $H$ be a facet-defining hyperplane of $\bar{S}$
    such that $F = S \cap H \neq \emptyset$. It is easy to see that $F$ is a
    half-open $(s-1)$-simplex with $M(F) < M(S)$ and we conclude
    \[
        \Cls{S} \ = \ \Cls{S \setminus F} + \Cls{F} \ \in \ \WCM
    \]
    by induction.
\end{proof}

Dissections of simplices into cylinders are widely used in connection with
simple valuations. As we do not want to neglect `lower dimensional' parts, our
results have to be more refined. The next lemma is key to the proof of
Theorem~\ref{thm:CM_weakly}.

\begin{lem} \label{lem:hosimplex_binom}
    Let $\HO{S} \subset \R^d$ be a half-open simplex of dimension $d >
    0$. Then
    \[
        \Cls{n\HO{S}} \ = \ 
         \zeta_0 + 
        \binom{n}{1} \zeta_1 + 
        \binom{n}{2} \zeta_1 + 
        \cdots + 
        \binom{n}{d} \zeta_d,
    \]
    where $\zeta_k \in \Zyl_{k}$ for all $k=0,\dots,d$. If $\HO{S}$ is
    properly half-open, then $\zeta_0 = 0$, otherwise $\zeta_0=\Cls{0}$.
\end{lem}
\begin{proof}
    If $\dim \HO{S} = 0$, then the result is trivial. Otherwise if $\HO{S}$ is closed and 
$\dim S = d >
    0$, we can appeal to the proof of Lemma~\ref{lem:Zyl_chain} and write
    $\HO{S} =
    S_0 \uplus S_1 \uplus \cdots \uplus S_d$ where $S_i$ is a properly
    half-open simplex of dimension $i$ if $i>0$ and $S_0$ is a vertex. Thus, it
    suffices to assume that $\HO{S}$
    is properly half-open. 

\newcommand\bp{\overline{p}}%
    Up to affine transformation, $nS \ = \ \{ x \in \R^d : 0 \le x_1 \le x_2
    \le \cdots \le x_d \le n \}.$
    Let $p \in nS$ be a generic point. Then $p$ is of the form $0 < p_1 <
    \cdots < p_d < n$. The coordinates of $\bp = ( \lfloor p_1 \rfloor,
    \lfloor p_2 \rfloor, \dots, \lfloor p_d \rfloor)$ are weakly increasing
    integers between $0$ and $n-1$ and thus the points $\bp$ are in bijection
    with a $d$-multisubset of $[n]$. For $\bp = (0^{a_1},1^{a_2}, \dots,
    (n-1)^{a_n})$, we write $(a_1,\dots,a_n)$ for the corresponding
    multisubset and $a_1 + \cdots + a_n = d$. In particular, 
    \[
        p - \bp \in
        S(\bp) \ := \ \left\{ x \in \R^d:
        \begin{array}{c}
            0 \le x_1 \le \cdots \le x_{b_{1}} \le 1 \\
            0 \le x_{b_1+1} \le \cdots \le x_{b_2} \le 1 \\
            \cdots\\
            0 \le x_{b_{n-1}+1} \le \cdots \le x_{b_n} \le 1
        \end{array}
        \right\},
    \]
    where $b_i := a_1 + \cdots + a_i$. Since every generic point $p$ gives
    rise to a unique such cell $\bp + S(\bp)$, it follows that 
    \[
        nS \ = \ \bigcup_{\bp \in (n-1)S \cap \Z^d} \bp + S(\bp)
    \]
    is a dissection. By construction $S(\bp)$ is a cylinder, i.e., a Cartesian
    product of $1 \le k \le d$ standard simplices, where $k$ is the number of
    $i$ with $a_i >0$. Using Proposition~\ref{prop:simplex_ihalf}, pick $u$
    generic with respect to $S$ such that $S=\HO{S} = \Ihalf{u}{S}$. It follows
    from Lemma~\ref{lem:ho} that  for any fixed $n$
    \[
        \Ihalf{u}{nS} \ = \ \biguplus_{\bp \in (n-1)S \cap \Z^d} \bp +
        \Ihalf{u}{S(\bp)}
    \]
    is a half-open dissection of $\HO{S}$ into half-open cylinders.
    Note that there are exactly $\binom{n}{k}$ many $k$-multisubsets
    $\bp' \in (n-1)S \cap \Z^d$ such that $S(\bp')=S(\bp)$ and hence
    \[
        \Cls{n\HO{S}} \ = \ 
        \binom{n}{1} \zeta_1 + 
        \binom{n}{2} \zeta_1 + 
        \cdots + 
        \binom{n}{d} \zeta_d,
    \]
    where $\zeta_k \in \Zyl_{k}$ for all $k=1,\dots,d$.
\end{proof}

A nice consequence of Lemma~\ref{lem:hosimplex_binom} is that we can scale the
individual summands in a cylinder.

\begin{lem} \label{lem:HOprod_binom}
    Let $\HO{S} = \HO{S}_1 + \cdots + \HO{S}_\Npolys \subset \R^d$ be a
    half-open cylinder.  Then 
    \[
        \Cls{n_1\HO{S}_1 + n_2\HO{S}_2 + \cdots + n_\Npolys\HO{S}_\Npolys} \ =
        \ \sum_{\k \in \Z^\Npolys_{\ge 0}} \zeta_{\k} \binom{n_1}{k_1}\cdots
        \binom{n_\Npolys}{k_\Npolys},
    \]
    for all $n_1,\dots,n_\Npolys \ge 0$ for some, where $\zeta_\k \in
    \Zyl_{|\k|}$.  If $\HO{S}$ is properly half-open, then $\zeta_{\0} = 0$.
\end{lem}
\begin{proof}
    For a linear subspace $U \subset \R^d$, let $\Zyl_k(U) \subset \Zyl_k$ be
    the subgroup generated by the half-open $k$-cylinders in $U$. If $U_1,U_2
    \subset \R^d$ are subspaces such that $U_1+U_2 = U_1 \oplus U_2$, then
    $\Zyl_{k_1}(U_1) \cdot \Zyl_{k_2}(U_2) \subseteq \Zyl_{k_1+k_2}(U_1+U_2)$.
    We can assume that $0 \in S_i$ for $i=1,\dots,m$ and let $U_i$ be the
    linear span of $S_i$. Since the sum $S_1 + \cdots + S_\Npolys$ is exact, it
    follows that 
    \[
        \Cls{n_1\HO{S}_1 + n_2\HO{S}_2 + \cdots + n_\Npolys\HO{S}_\Npolys}
        \ = \ 
        \Cls{n_1\HO{S}_1} \cdot \Cls{ n_2\HO{S}_2} \cdots
        \Cls{n_\Npolys\HO{S}_\Npolys}
    \]
    and the claim follows from Lemma~\ref{lem:hosimplex_binom}.
\end{proof}

\subsection{Fine mixed half-open dissections}
\newcommand\fmSub{\mathcal{M}}
\newcommand\nn{\mathbf{n}}

Let $P_1,\dots,P_\Npolys \subset \R^d$ be nonempty polytopes such that $P :=
P_1 + \cdots + P_\Npolys$ is of dimension $k$. A \Defn{fine mixed dissection}
of $P$ is a dissection
\[
    P \ = \ R_1 \cup R_2 \cup \cdots \cup R_m
\]
in which each $R_i$ is a cylinder of the form $R_i = R_{i1} + \cdots +
R_{i\Npolys}$ such that $R_{ij} \subseteq P_j$ is a $\LL$-polytope. In order
to find such a fine mixed dissection, we can use the well-known Cayley trick
adapted to dissections: Define the Cayley polytope 
\[
    \Cay(P_1,\dots,P_\Npolys) \ = \ \conv\left( \bigcup\nolimits_i P_i \times \{e_i\}\right) \ \subset \R^d
    \times \R^\Npolys,
\]
where we write $e_1,\dots,e_\Npolys$ for the standard basis of $\R^\Npolys$.
Let $W = \{ (x,y) \in \R^d \times \R^\Npolys : y_i = \frac 1 \Npolys \text{ for }
i=1,\dots, \Npolys \}$. Then
$\Cay(P_1,\dots,P_\Npolys) \cap W \cong \frac{1}{\Npolys}(P_1 + \cdots + P_\Npolys)$. Of course,
we may intersect a dissection of $P_1+\cdots+P_\Npolys$ and inspect the pieces
separately. 
\begin{lem}[{Cayley trick~\cite[Sect.~9.2]{DLRS}}]
    Fine mixed dissections of $P_1 + \cdots + P_\Npolys$ are in bijection to
    dissections of $\Cay(P_1,\dots,P_\Npolys)$.
\end{lem}

For $\nn = (n_1,\dots,n_\Npolys) \in \Z_{>0}^\Npolys$, we write $\nn P =
n_1P_1 + \cdots + n_\Npolys P_\Npolys$. The following is a consequence of the
Cayley trick.
\begin{cor}
    Let $P = P_1 + \cdots + P_\Npolys  = R_1 \cup R_2 \cup \cdots \cup R_m$ be a
    fine mixed dissection. Then for every $\nn \in \Z_{>0}^\Npolys$
    \[
        \nn P \ = \ \nn R_1 \cup \nn R_2 \cup \cdots \cup \nn R_m.
    \]
    is a fine mixed dissection.
\end{cor}

With this, we can give the main result of this section.
\begin{thm}\label{thm:diff_zyl}
    Let $P_1,\dots,P_\Npolys, Q_1,\dots,Q_\Npolys \subset \R^d$ be polytopes
    with $P_i \subseteq Q_i$ for all $i=1,\dots,m$. Then
    \[
    \Cls{n_1Q_1 + \cdots + n_\Npolys Q_\Npolys} -
    \Cls{n_1P_1 + \cdots + n_\Npolys P_\Npolys} 
    \ = \ \sum_{\k \in \Z^\Npolys_{\ge0}} \zeta_\k \binom{n_1}{k_1} \cdots
    \binom{n_\Npolys}{k_\Npolys},
    \]
    where $\zeta_\k \in \Zyl_{|\k|}$.

\end{thm}
\begin{proof}
    Note that $\Cay(P_1,\dots,P_\Npolys ) \subseteq \Cay(Q_1,\dots,Q_\Npolys )$. Let us
    first assume that $Q = Q_1 + \cdots + Q_\Npolys $ and $P = P_1 + \cdots + P_\Npolys $
    are of the same dimension.  Using \emph{placing
    triangulations}~\cite[Sect.~4.3.1]{DLRS} and the Cayley trick,  we can
    thus find a fine mixed dissection
    \[
        Q \ = \  R_1 \cup \cdots \cup R_N
    \]
    such that 
    \[
        P \ = \  R_{M+1} \cup \cdots \cup R_N
    \]
    for some $M < N$. Let $q \in \relint(P)$ be a generic point. Then 
    \[
        Q \setminus P \ = \ 
        \half{q}{R_1} \uplus \cdots \uplus
        \half{q}{R_M}
    \]
    and hence 
    \[
        \Cls{\nn Q} - \Cls{\nn P} \ = \ \sum_{i=1}^M \Cls{\half{q}{\nn R_i}}.
    \]
    The result then follows from Lemma~\ref{lem:HOprod_binom}.
    
    Now assume that $s := \dim Q - \dim P > 0$. Then there are
    $\Lambda$-polytopes
    \begin{eqnarray*}
        P_1   & \ = \ &  P_1^0 \ \subseteq \ P_1^1\subseteq \cdots \subseteq
        P_1^s \ \subseteq \ Q_1\\
            \ &\vdots& \ \\
        P_\Npolys    & \ = \ &  P_\Npolys ^0 \ \subseteq \ P_\Npolys ^1\subseteq \cdots \subseteq
        P_\Npolys ^s \ \subseteq \ Q_\Npolys 
    \end{eqnarray*}
    such that $\Cay(P_1^i,\dots,P_\Npolys ^i)$ is a pyramid over
    $\Cay(P_1^{i-1},\dots,P_\Npolys ^{i-1})$ with apex $a^i \in (P_l^i\setminus
    P_l^{i-1}) \cap \LL$ for some $l$, and $P_1^s+\cdots +P_\Npolys ^s$ is of the
    same dimension as $Q$.  In particular,
    $\Cay(P_1^{i-1},\dots,P_\Npolys ^{i-1})\subset \Cay(P_1^i,\dots,P_\Npolys ^i)$ is a
    facet, and by the Cayley trick, also $P^i:=P_1^{i-1}+\cdots +P_\Npolys ^{i-1}
    \subset P_1^{i}+\cdots +P_\Npolys ^{i}$ is a facet. For a fine mixed dissection
    \[
        P_1^{i}+\cdots +P_\Npolys ^{i} \ = \  R_1 \cup \cdots \cup R_M
    \]
    we obtain 
    \[
        \Cls{\nn P^i} - \Cls{\nn P^{i-1}} \ = \ \sum_{i=1}^M \Cls{\half{q}{\nn
        R_i}}.
    \]
    by choosing a point $q$ beyond the facet $P_1^{i-1} + \cdots + P_\Npolys ^{i-1}$
    and beneath all others. Since $1\leq i\leq s$ was chosen arbitrarily,
    together with the first part, the proof follows.
\end{proof}

Within the framework developed in Section~\ref{ssec:cones}, we can finally
prove Theorem~\ref{thm:CM_weakly}.

\begin{proof}[Proof of Theorem~\ref{thm:CM_weakly}]
    Let $\phi : \PolL \rightarrow G$ be a weakly $h^*$-monotone valuation and
    let $P_1 \subseteq Q_1, \dots, P_\Npolys \subseteq Q_\Npolys$ be
    $\LL$-polytopes. 
    Then
    \[
        \CM\phi(Q_1,\dots,Q_\Npolys) - \CM\phi(P_1,\dots,P_\Npolys) \ = \
        \extPhi( \CM\Cls{ Q_1,\dots, Q_\Npolys} - \CM\Cls{ P_1,\dots,
        P_\Npolys})
    \]
    by Corollary~\ref{cor:CMcls}. The class $\CM\Cls{ Q_1,\dots, Q_\Npolys} -
    \CM\Cls{ P_1,\dots, P_\Npolys}$ is the coefficient $\zeta_{(1,\dots,1)}$ of
    $\binom{n_1}{1}\cdots\binom{n_\Npolys}{1}$ in Theorem~\ref{thm:diff_zyl}.
    By Lemma~\ref{lem:Zyl_chain}, we have $\zeta_{(1,\dots,1)} \in \WCM$ and
    hence ${\extPhi(\zeta_{(1,\dots,1)}) \ge 0}$.
\end{proof}

On the level of cones in the polytope algebra, Theorem~\ref{thm:CM_weakly}
states the following.

\begin{cor}\label{cor:mainCone}
    $
    \overline{\M} \ \subseteq \
        \Zyl_1 \ = \ \WCM.
    $
\end{cor}

\subsection{Lower bounds on combinatorial monotone valuations}
\label{ssec:bounds}
In the classical setting, there is an elementary geometric characterization
for the positivity of mixed volumes. The following result extends this
characterization to combinatorial mixed volumes.
\begin{thm}\label{thm:trivialitycond}
    Let $\phi  \colon \PolL \rightarrow G$ be a combinatorially positive
    $\LL$-valuation with $\phi(\{0\})>0$. Then for $\LL$-polytopes
    $P_1,\ldots, P_r$ the following are equivalent:
    \begin{enumerate}[\rm (i)]
        \item $\CM_r\phi(P_1,\dots,P_r)> 0$;
        \item  There exist linearly independent segments $S_1\subseteq
            P_1,\dots, S_r \subseteq P_r$  with vertices in $\LL$.
    \end{enumerate}
\end{thm}
\begin{proof}
    (ii) $\Longrightarrow$ (i): Since $\CM_r\phi$ is monotone in each argument
    we have 
    \[
        \CM_r\phi(P_1,\dots,P_r) \ \geq \ \CM_r\phi(S_1,\dots,S_r) \ = \ 
        \phi(\Ihalf{u}{(S_1+\dots +S_r)}),
    \]
    where the last equality holds by~\cite[Prop. 2.8]{HKST} for any general $u
    \in \R^d$. Observe that
    \[
        \Cls{\Ihalf{u}{(S_1+\dots +S_r)}} \ = \  \sum _{I\subseteq
        [r]}\Cls{\relint{\left(S_I\right)}},
    \]
    and therefore
    \[
        \CM_r\phi(P_1,\dots,P_r) \ \geq \ \sum_{I\subseteq 
        [r]}\phi(\relint S_I ) \ \geq \  \phi(\{0\}) \ > \ 0.
    \]
    (i) $\Longrightarrow$ (ii): By Corollary~\ref{cor:CMcls} and
    Theorem~\ref{thm:diff_zyl}, $\CM\Cls{P_1,\dots,P_r} \in \Zyl_r$. From the
    assumption $\CM_r\phi(P_1,\dots,P_r)> 0$, it follows from the proofs of
Lemma \ref{lem:HOprod_binom} and Theorem \ref{thm:diff_zyl} that there has to be
    at least one $r$-cylinder $T = T_1 + \cdots +
    T_r$, where $T_i \subset P_i$ is positive dimensional for $i=1,\dots,r$.
    Thus, choosing $S_i \subseteq T_i$ to be an edge for $i = 1,\dots,r$ proves
    the claim.
\end{proof}

We cannot weaken the assumption of Theorem~\ref{thm:trivialitycond} and assume
that $\phi$ is only \emph{weakly} $h^*$-monotone. To see this, consider the
dissection of the half-open $2$-cylinder $[0,1)^2$ into half-open simplices
\[
    [0,1)^2  \ = \ \{x\in [0,1)^2\colon x_1<x_2\} \uplus \{x\in [0,1)^2\colon
    x_1>x_2\} \uplus \{x\in [0,1)^2\colon x_1=x_2\}.
\]
Using Theorem~7.4 in~\cite{JS15}, we can manufacture a weakly $h^*$-monotone
$\Z^d$-valuation $\phi$ such that $\phi(\{0\})>0$ and $\phi([0,1)^2) = 0$.

By way of Theorem~\ref{thm:trivialitycond}, combinatorial positivity in
particular implies that the computational task of deciding whether
$\CM\phi(P_1,\dots,P_\Npolys) > 0$ can be decided in polynomial time. For the
mixed volume, this was achieved by Dyer, Gritzmann, and
Hufnagel~\cite[Theorem~8]{DGH}. In fact, the proof of this result uses the
fact that the mixed volume is monotone and positive if there are sufficiently
many linearly independent segments. This can be phrased as a \emph{matroid
intersection} problem of two matroids, which can be solved in polynomial time.
The proof of Theorem 8 in~\cite{DGH} carries over to our case and shows the
following.
\begin{cor}\label{cor:complex}
    Let $\phi$ be a combinatorial positive $\LL$-valuation with $\phi(\{0\})
    >0$ and $P_1,\dots,P_\Npolys$ well-presented $\LL$-polytopes. Then there
    is a polynomial time algorithm that decides whether
    \[
        \CM\phi(P_1,\dots,P_\Npolys) \ = \ 0.
    \]
\end{cor}

For the discrete volume, we can give stronger lower bounds on $\CM\dvol$.
\begin{cor}\label{cor:prod_simp}
    Let $P_1,\dots,P_r \in \Pol{\Z^d}$ be lattice polytopes. For any choice of
    lattice simplices $S_i \subset P_i$ for $i=1,\dots,r$ such that $S_1 +
    \cdots + S_r$ is an $r$-cylinder, we have
    \[
        \CM_r\dvol(P_1,\dots,P_r) \ \geq  \ \dim S_1 \cdots \dim S_r.
    \]
\end{cor}
\begin{proof}
    By monotonicity $\CM_r\dvol(P_1,\dots,P_r) \geq
    \CM_r\dvol(S_1,\dots,S_r)$.  Since $S = S_1+\cdots + S_r$ is exact and
    therefore affinely isomorphic to $S_1 \times \cdots \times S_r$, $S_I$
    is a face of $S$ for all $I$. Proposition~2.8 in~\cite{HKST} then yields
    \[
        \CM_k\dvol(S_1,\dots,S_\Npolys) \ = \ \dvol\left(S\setminus
        \bigcup_{I\subset [\Npolys]}S_I\right) \ \geq \ \dim S_1 \cdots \dim
        S_\Npolys,
    \] 
    since every lattice polytope of dimension $l$ has at least $l+1$ lattice
    points.
\end{proof}

\subsection{Discrete volume and Cayley cones}
\label{ssec:cayley}

In this section, we briefly outline an alternative proof of
Corollary~\ref{cor:Bihan} that does not make use of the polytope algebra.  For
given (lattice) polytopes $P_1,\dots,P_r \subset \R^d$, define the
\Defn{Cayley cone} 
\begin{equation}\label{eqn:cayley}
    C \ = \ \{ (p,\l) \in \R^d \times
    \R^r : \l\ge 0, p \in \l_1 P_1 + \cdots + \l_r P_r
    \}\,,
\end{equation}
i.e., $C = \cone (\Cay(P_1,\dots,P_r))$. As
for the Cayley polytope, we observe that for fixed $ n_1,\dots,n_r \in \Znn$
\[
    n_1 P_1 + \cdots + n_r P_r \ = \ \{ p \in \R^d : (p,n_1,\dots,n_r) \in 
    C  \}.
\]
The \Defn{integer-point transform} of a set $S \subseteq \R^d \times \R^r$
is the formal Laurent series
\[
    \sigma_S(\x,\y) \ := \ \sum_{(p,\l) \in S \cap (\Z^d \times \Z^r)} \x^p
    \y^\l\,.
\]
Hence, the integer-point transform $\sigma_C(\x,\y)$ is the sum of all Laurent
monomials $\x^py_1^{n_1}\cdots y_r^{n_r}$ where $n_1,\dots,n_r \in \Znn$ and
$p$ is a lattice point of $n_1P_1 + \cdots + n_r P_r$. To compute
$\sigma_C(\x,\y)$, let $C = S_1 \cup \cdots \cup S_m$ be a dissection of $C$
into simplicial cones such that the generators of $S_k$ are among the
generators of $C$ for each $k$.  For a point $q$ in the interior of $C$ and
generic with respect to all $S_k$, the half-open decomposition of $C$ with
respect to $q$ yields
\[
    \sigma_C(\x,\y) \ = \ \sigma_{\half{q}{S_1}}(\x,\y) + \cdots +
    \sigma_{\half{q}{S_m}}(\x,\y).
\]
Computing integer-point transforms of half-open simplicial cones is quite
straightforward (and explicitly done in~\cite[Section~4.6]{crt}): For every
$1 \le k \le m$ there is a bounded set $\mathcal{P}_k \subset \Z^d \times
\Z^r$, points $v_{k,1}, \dots, v_{k,d+r} 
\in \Z^d$, and $1 \le j_{k,1} \le \cdots \le j_{k,d+r} \le r$ such that 
\[
    \sigma_{\half{q}{S_h}}(\x,\y) \ = \
    \frac{\sigma_{\mathcal{P}_i}(\x,\y)}{ (1-\x^{v_{k,1}}y_{j_{k,1}})
    \cdots (1-\x^{v_{k,d+r}}y_{j_{k,d+r}}) }\,.
\]
In particular, setting $\x = \1$
\begin{equation}\label{eqn:multiH}
    \sum_{\nn \in \Znn^r} E(n_1P_1 + \cdots + n_rP_r) y_1^{n_1}\cdots
    y_r^{n_r} \ = \ 
    \sum_{h=1}^m
    \frac{\sigma_{\mathcal{P}_h}(\1,\y)}
    {(1-y_1)^{d_{h,1}+1} \cdots (1-y_r)^{d_{h,r}+1} }
\end{equation}
then shows that $E(n_1P_1 + \cdots + n_rP_r)$ agrees with a polynomial.  

Since
each $S_k$ is a simplicial cone of full dimension $d+r$, it follows that $0
\le d_{k,i} \le \dim P_k$ and $\sigma_{\mathcal{P}_k}(\1,\y)$ is of degree
$\le d_{k,i}$ in $y_i$. Unravelling the right-hand side
of~\eqref{eqn:multiH} shows that $E(n_1P_1 + \cdots + n_rP_r)$ is a 
nonnegative\footnote{Note that the coefficients of
$\sigma_{\mathcal{P}_h}(\1,\y)$ are clearly nonnegative. However, writing the
right-hand side of~\eqref{eqn:multiH} with denominator $ (1-y_1)^{\dim
P_1+1}\cdots (1-y_1)^{\dim P_r+1}$, the coefficients of the numerator are not
necessarily nonnegative.}
linear combination of the polynomials
\[
\binom{n_1 + d_{k,1} -
j_{k,1}}{d_{k,1}} 
\binom{n_2 + d_{k,2} -
j_{k,2}}{d_{k,2}} 
\cdots \binom{n_r + d_{k,r} - j_{k,r}}{d_{k,r}}
\]
with $0\le j_i \le d_{k,i}$ for all for $k=1,\dots,m$ and $i = 1,\dots,r$. 
Rewriting this in the form of Theorem~\ref{thm:repr} proves that
\[
    \CM_rE(P_1,\dots,P_r) \ = \ \sum_{I \subseteq [r]} (-1)^{r-|I|} E(P_I) \
    \ge \ 0.
\]
To prove monotonicity, one sees that $C$ is a subcone of $C' =
\cone(\Cay(Q_1,\dots,Q_r))$ for lattice polytopes $Q_i \supseteq P_i$. The
same ideas as in the proof of Theorem~\ref{thm:diff_zyl} show that 
$\sigma_{C'}(\x,\y) - \sigma_{C}(\x,\y)$ is a sum of integer-point
transforms of half-open simplicial cones and the above argument yields the
claim.

\section{Monotonicity and CM-monotonicity}
\label{sec:monotone}

Let $\phi : \PolL \rightarrow G$ be a $\LL$-valuation. If $\phi$ is
CM-monotone, then in particular
\[
    \CM_1\phi(Q) \ = \ \phi(Q) - \phi(\{0\}) \ \ge \ \phi(P) - \phi(\{0\})
    \ = \ \CM_1\phi(P),
\]
for any two $\LL$-polytopes $P \subseteq Q$. Hence CM-monotonicity implies
monotonicity. In this section we explore the converse implication. As
in~\cite{JS15}, there is a fundamental difference between the cases $\LL =
\R^d$ and $\LL = \Z^d$. To highlight the difference, we put the following
well-known lemma on record.  Let $\Konv_d$ denote the family of convex bodies
in $\R^d$.
\begin{lem}\label{lem:Kext}
    Let $\phi : \Pol{\R^d} \rightarrow \R$ be a monotone $\R^d$-valuation. Then
    $\phi$ can be extended to a monotone and translation-invariant valuation
    on $\Konv_d$.
\end{lem}

It was shown in~\cite[Prop.~5.4]{JS15} that every weakly $h^*$-monotone
valuation is monotone and Corollary~\ref{cor:mainCone} suggests that
CM-monotone and weakly $h^*$-monotone valuations may coincide. This, however, is
not true for $\LL = \R^d$; see Corollary~\ref{cor:not_weak} below.

The classical mixed volumes play a central role in Hadwiger's characterization
of continuous and rigid motion invariant valuations on convex bodies. We can
use specializations of mixed volumes to manufacture CM-monotone valuations.
\begin{prop}\label{prop:MV_CMmono}
    Let $Q_1,\dots,Q_{d-\Npolys} \subset \R^d$ be nonempty polytopes for $0 \le \Npolys \le
    d$. Then 
    \[
        P \ \mapsto \ \phi(P) \ := \ \MV_d(P^{\Npolys},Q_1,...,Q_{d-\Npolys})
    \]
    is a CM-monotone $\R^d$-valuation.
\end{prop}
\begin{proof}
    Using the multilinearity of the mixed volume we obtain
    \[
        \phi(n_1P_1 + \cdots + n_\Npolys P_\Npolys ) \ = \ \sum _{i_1,\ldots,
        i_\Npolys=1}^\Npolys \MV_d(P_{i_1}, \dots, P_{i_\Npolys},Q_1,\dots, Q_{d-\Npolys})n_{i_1}\cdots
        n_{i_\Npolys}.
    \]
    Therefore, the coefficients of $\phi_{P_1,\ldots, P_\Npolys}$ in the
    standard monomial basis are monotone valuations in each argument $P_i$.
    By~\cite[Thm~2.2]{HKST}, $\CM_k\phi$ is a nonnegative linear combination
    of these coefficients for all $k$ and, as such, monotone as well.
\end{proof}

Together with Hadwiger's characterization theorem
(cf.~\cite[Thm.~6.4.14]{SchneiderNew}), we get a characterization of
CM-monotone 
rigid-motion invariant $\R^d$-valuations.
\begin{cor}\label{cor:rigidmotionCM}
    A rigid-motion invariant valuation $\phi : \Pol{\R^d} \rightarrow \R$ is
    monotone if and only if it is CM-monotone. In particular, all intrinsic
    volumes (or quermassintegrals) are CM-monotone.
\end{cor}
\begin{proof}
    By Lemma~\ref{lem:Kext}, we may assume that  $\phi : \Konv_d \rightarrow
    \R$ is rigid-motion invariant and monotone. Thus 
    \[
        \phi(K) \ = \ c_0 \nu_0(K) + \cdots + c_d \nu_d(K)
    \]
    for $\nu_i = \MV_d(K^i,B_d^{d-i})$ and some $c_0,\dots,c_d \ge 0$.  An
    appeal to continuity and to Proposition~\ref{prop:MV_CMmono} completes the
    proof.
\end{proof}

In~\cite{JS15} it was shown that $\nu_0 = \chi$ and $\nu_d = \vol$ are the
only weakly $h^*$-monotone quermassintegrals. This sorts out the relation
between weakly $h^*$-monotone and CM-monotone valuations for $\LL = \R^d$.
\begin{cor}\label{cor:not_weak}
    If $d>1$, then the inclusion $\MM  \subseteq \WCM$ is strict.
\end{cor}

Of course the restriction to valuations invariant under rigid motions is
rather strong. Bernig and Fu~\cite{BernigFu} proved a deep result about
monotonicity of translation-invariant valuations and its homogeneous
components that was stated in the introduction.

\begin{reptheorem}{thm:BF}
    Let $\phi : \Konv_d \rightarrow \R$ be a translation-invariant valuation
    and let $\phi = \phi_0 + \cdots + \phi_d$ be the decomposition into
    homogeneous components. Then $\phi$ is monotone if and only if $\phi_i$ is
    monotone for all $i=0,\dots,d$.
\end{reptheorem}

The next lemma will allow us to show CM-monotonicity for general monotone
$\R^d$-valuations.

\begin{lem}\label{lem:bernig}
    Let $\phi : \Konv_d \rightarrow \R$ be a translation-invariant valuation
    homogeneous of degree $\Npolys$ and $\Mix_\Npolys\phi : (\Konv_d)^\Npolys \rightarrow \R$
    the corresponding mixed valuation. Then $\phi$ is monotone if and only if
    $\Mix_\Npolys\phi$ is monotone in each component.
\end{lem}
\begin{proof}
    For $\Npolys=1$, this is certainly true. For $K_1,\dots,K_\Npolys \in \Konv_d$, write
    \[
        \phi(\l_1K_1 + \cdots + \l_\Npolys K_\Npolys) \ = \ \sum_{|\Ba|=\Npolys}
        \binom{\Npolys}{\a_1,\dots,\a_\Npolys}
        \Mix_\Npolys\phi(K_1^{\a_1},\dots,K_\Npolys^{\a_\Npolys})\l^\a.
    \]
    For fixed $K_\Npolys$, the translation-invariant valuation $\phi^{+K_\Npolys}(K) :=
    \phi(K + K_\Npolys)$ is monotone. By Theorem~\ref{thm:BF}, its homogeneous
    components $\phi_i^{+K_\Npolys}$ are monotone for all $i$ and the mixed
    valuation of $\phi_{\Npolys-1}^{+K_\Npolys}$ satisfies
    \[
        \Mix_{\Npolys-1}\phi^{+K_\Npolys}_{\Npolys-1}(K_1,\dots,K_{\Npolys-1})
        \ = \ \Npolys \,
        \Mix_{\Npolys}\phi(K_1,\dots,K_{\Npolys-1},K_\Npolys),
    \]
    which is monotone in each $K_1,\dots,K_{\Npolys-1}$ by induction on $\Npolys$. Since
    $\Mix\phi$ is symmetric, this shows monotonicity in each argument.
\end{proof}

\begin{reptheorem}{thm:Rd-mono}
    Let $\phi : \Pol{\R^d} \rightarrow \R$ be a translation-invariant 
    valuation. Then $\phi$ is monotone if and only if it is CM-monotone.
\end{reptheorem}
\begin{proof}
    If $\phi$ is CM-monotone, then $\phi$ is monotone.  For the converse,  we
    extend $\phi$ to a translation-invariant and monotone valuation on
    $\Konv_d$ by using Lemma~\ref{lem:Kext}.  It follows from the proof of
    Theorem~2.2 in~\cite{HKST} that
    \[
        \CM_\Npolys\phi(P_1,\dots,P_\Npolys) \ = \ \sum_{\Ba}
        \binom{|\Ba|}{\a_1,\dots,\a_\Npolys} \Mix_\Npolys\phi(P_1^{\a_1},\dots,P_\Npolys^{\a_\Npolys}),
    \]
    where the sum is over all $\Ba \in \Z^\Npolys_{\ge 1}$ and, by
    Lemma~\ref{lem:bernig}, $\CM_\Npolys\phi$ is a nonnegative linear
    combination of monotone valuations. 
\end{proof}

Theorem \ref{thm:Rd-mono} suggests the conjecture stated in the introduction.
\begin{repconj}{conj:mono}
    Let $\phi$ be a $\Lambda$-valuation for any $\LL$. Then $\phi$ is monotone
    if and only if $\phi$ is CM-monotone.
\end{repconj}

In support of the conjecture, we can show the following. A valuation $\phi$ is
\Defn{simple} if $\phi(P) = 0$ for all $P \in \PolL$ with $\dim P < \dim \LL$.

\begin{prop}\label{prop:DM_simple}
        If $\phi$ is a simple and nonnegative $\LL$-valuation, then $\phi$ is
        CM-monotone.
\end{prop}
\begin{proof}
    If $\phi$ is simple and nonnegative, then
    \[
    \phi(\relint P) + \phi(\relint F) \ = \ 
    \phi(\relint P) \ = \ 
    \phi(P) \ \ge \ 0,
    \]
    holds for any $\LL$-polytope $P$ and facet $F \subset P$. Hence $\phi$ is
    weakly $h^*$-monotone and the result follows from
    Corollary~\ref{cor:mainCone}.
\end{proof}

An idea used in the proof of Lemma~\ref{lem:bernig} also furnishes a
supposedly large class of CM-monotone valuations to draw from.

\begin{prop}\label{prop:addpolytope}
    Let $\phi$ be a CM-monotone $\LL$-valuation and $Q \in \PolL$. Then 
    \[
        P \ \mapsto \ \phi^{+Q}(P) \ := \ \phi(P+Q)
    \]
    is a CM-monotone $\LL$-valuation.
\end{prop}
\begin{proof}
    Observe that 
    \[
        \CM_\Npolys\phi^{+Q}(P_1,\dots,P_\Npolys) \ = \ \CM_{\Npolys+1}\phi(P_1,\dots,P_\Npolys,Q)  \
        + \ \CM_\Npolys\phi(P_1,\dots,P_\Npolys) 
    \]
    and hence $\CM_\Npolys\phi^{+Q}$ is a sum of valuations monotone in each
    argument $P_i$.
\end{proof}

For further evidence towards the validity of Conjecture~\ref{conj:mono}, we
finally show that it would imply the Bernig--Fu Theorem~\ref{thm:BF}.

The Euler characteristic $\chi$ yields a ring map $\chi : \Pi \rightarrow \Z$.
If $\LL$ is not a lattice, then it was shown by McMullen~\cite{mcmullenPA}
that $\Pi_+ := \{ x \in \PiL : \chi(x) = 0 \}$ is a $\Q$-vector space spanned
by $\Cls{P} - 1$, $P \in \PolL$. Since $\Cls{P}$ is unipotent, we can define
the \Defn{logarithm} of a polytope $P \in \PolL$ as
\[
    \log P \ := \ \log( 1 + (\Cls{P} - 1)) 
    \ = \ \sum_{k = 1}^d \tfrac{(-1)^k}{k} (\Cls{P} - 1)^k
    \ = \ \sum_{k = 1}^d \tfrac{(-1)^k}{k} \CM_k\Cls{P,...,P}.
\]
It follows that 
\begin{equation}\label{eqn:exp}
    \Cls{n P}  \ = \ \Cls{P}^n \ = \ \exp(n\log P) \ := \ 1 + \sum_{i=1}^d
    \tfrac{\log(P)^i}{i!} n^i.
\end{equation}
For $b \in \Z_{>0}$, we note that $\Cls{P} = \Cls{\frac{1}{b} P}^b$ and
therefore $\log(\frac{1}{b}P) = \frac{1}{b} \log(P)$. Hence, $\Cls{P}^q =
\Cls{qP}$ holds for all $q \in \Q_{>0}$.

\newcommand\Mh[1]{\M^{\langle #1 \rangle}}%
For a $\LL$-valuation $\phi$, let $\phi = \phi_0 + \cdots + \phi_d$ be its
decomposition into homogeneous parts. Then $i!\phi_i(P) = \phi(\log(P)^i)$.
For $1 \le i \le d$ define
\[
    \Mh{i} \ := \ 
    \Znn \left\{ \log(Q)^i - \log(P)^i : P,Q \in \PolL, P \subseteq Q \right\} \ \subset
    \ \PiL_+ \, .
\]
The nontrivial implication of Theorem~\ref{thm:BF} is equivalent to
\[
    \Mh{i} \ \subseteq \ \M
\]
for all $i = 1,\dots,d$. Theorem~\ref{thm:CM-BF} follows directly from the
next result.

\begin{thm}\label{thm:CM_log}
    Let $\LL \subseteq \R^d$ be a $\Q$-vector space of dimension $d$. Then
    \[
        \Mh{i} \ \subseteq \MM
    \]
    for all $1 \le i \le d$.
\end{thm}
\begin{proof}
    Let $P$ be a nonempty $\LL$-polytope and write $p = \log(P)$. Then
    $\Cls{P} = 1 + p + \cdots + \frac{p^d}{d!}$ and hence for any $\eps > 0$
    \[
    \eps^{-i} (\Cls{\eps P} - 1)^i \ = \ p^i + \eps^i
        r_P(\eps)
    \]
    where $r_P(\eps)$ is a polynomial in $\eps$ with coefficients in $\PiL_+$.
    Thus, for any $\LL$-valuation $\phi : \PolL \rightarrow \R$, we have that
    if $\phi \in \hom_+(\MM,\R)$, then 
    \[
        0 \ \le \ 
        \eps^{-i} \left( \CM_i\phi(\eps Q^i) - \CM_i\phi(\eps P^i) \right)
        \ = \ i!(\phi_i(Q) - \phi_i(P)) + \eps (\phi(r_Q(\eps)) -
        \phi(r_P(\eps)))
    \]
    for all $\LL$-polytopes $P \subseteq Q$ and all $\eps > 0$. Hence $\phi_i(Q)
    - \phi_i(P) \ge 0$ and $\phi \in \hom_+(\Mh{i},\R)$.
\end{proof}

If $\LL$ is a lattice, then $\PiL_+$ is not a $\Q$-vector space. Hence,  the
above arguments do not apply and other techniques will be needed. In the case
$\LL = \Z^2$, we know that $\CM_2\phi(P_1,P_2) = \rho(\phi) \MV(P_1,P_2)$ for
some $\rho(\phi) \in \R$ and $\CM_1\phi(P) = \phi(P) - \phi(\{0\})$.
Therefore, if $\phi$ is CM-monotone, then $\rho(\phi) \ge 0$ and
Conjecture~\ref{conj:mono} holds for $\Z^2$.

\bibliographystyle{siam}
\bibliography{DiscreteMixedValuations}

\end{document}